\documentclass[12pt]{amsart}
\usepackage{amsmath,amssymb, euscript, graphicx, hyperref, verbatim}

\newtheorem{thm}{Theorem}[section]
\newtheorem{lem}[thm]{Lemma}
\newtheorem{prop}[thm]{Proposition}
\theoremstyle{definition}
\newtheorem{defn}[thm]{Definition}

\newtheorem{rem}[thm]{Remark}

\newcommand{\blackboard}[1]{\ensuremath{\mathbb{#1}}}

\newcommand{\Z}{\blackboard{Z}}

\newcommand{\R}{\blackboard{R}}

 \title{Examples of left-orderable and non-left-orderable HNN extensions} 
 \author{Azer Akhmedov, Cody Martin}
 
 \address{Azer Akhmedov, Department of Mathematics,
North Dakota State University,
Fargo, ND, 58102, USA}
\email{azer.akhmedov@ndsu.edu}

\address{Cody Martin, Department of Mathematics,
North Dakota State University,
Fargo, ND, 58102, USA}

 
\begin{document}

 \maketitle

 \begin{abstract} We present a broad class of groups that contains all torsion-free nilpotent groups and has the property that, for any group in this class, its HNN extension is left-orderable. We also construct examples of non-left-orderable HNN extensions of left-orderable groups.
 \end{abstract}

 \section{Non-left-orderable HNN extensions of left-orderable groups}
 
 It is well-known that an HNN extension of a torsion-free group is still torsion-free (\cite{LS}, \cite{Bu}). On the other hand, for many classes of groups, existence of a torsion element is the only obstruction to left-orderability; for example, this is the case for the classes of one-relator groups, nilpotent groups, etc. Hence, it is natural to study how left-orderability behaves under an HNN extension.
 
 \medskip
 
 In \cite{BG} (see Example 6.2 there), an example is constructed to show that left-orderability is not preserved under the HNN extension. The example of \cite{BG} is built as an HNN extension of a direct product of a free nilpotent group of class two with the fundamental group of Klein bottle. Thus, it is an HNN extension of a group ``very close'' to a nilpoent one. {\em The major result of this paper} (Theorem \ref{thm:nilpotent}) is to show that such an example (construction)  is impossible for HNN extensions of groups from a certain very broad class which includes torsion-free nilpotent groups. In the opening section, we also present systematic ways of producing non-left-orderable HNN extensions of left-orderable groups. We produce examples of HNN extensions of groups such as non-Abelian free groups and virtually Abelian groups; these examples complete the discussion around the main result (Theorem \ref{thm:nilpotent}) besides appearing to be interesting to us independently. We rely on the following well-known criterion about left-orderability of groups \cite{N}
 
 \begin{prop} \label{thm:criterion} A group $G$ is left-orderable if and only if for all $k\geq 1$ and for all $g_1, \dots , g_k\in G\backslash \{1\}$, there exist $\epsilon _1, \dots , \epsilon _k\in \{-1,1\}$ such that the semigroup of $G$ generated by $g_1^{\epsilon _1}, \dots , g_k^{\epsilon _k}$ does not contain the identity element.
 \end{prop}
 
 Let us emphasize that we use the obvious ``only if part" of this proposition; the harder ``if part" is not needed.
 
 Given a group $G$ and subgroups $A, B\leq G$ with an isomorphism $\phi :A\to B$, the HNN extension $(G, A, B, t, \phi )$ is defined as the quotient of the free product $G\ast \langle t \rangle $ by the normal closure of the subset $\{tat^{-1}\phi (a)^{-1} \ | \ a\in A\}$. We also write this HNN extension as $(G, A, B, t)$ when $\phi $ is given in the context.

 \begin{prop}\label{them:free} A free group of rank bigger than one admits a non-left-orderable HNN extension.
 \end{prop}
 
 \begin{proof} By Britton's Lemma, it suffices to prove the theorem for the group $\mathbb{F}_2$. Let $a, b$ be the generators of $\mathbb{F}_2$. We can find positive exponents $p_i, q_i, r_i, s_i,  1\leq i\leq 8$ such that the elements $$u_1 = a^{p_1}b^{q_1}, u_2 = a^{p_2}b^{q_2}, u_3 = a^{p_3}b^{q_3}, u_4 = a^{p_4}b^{q_4},$$ \ $$u_5 = a^{p_5}b^{-q_5}, u_6 = a^{p_6}b^{-q_6}, u_7 = a^{p_7}b^{-q_7}, u_8 = a^{p_8}b^{-q_8}$$ generate a free group of rank 8, and so do the elements $$v_1 = a^{r_1}b^{s_1}, v_2 = a^{r_2}b^{-s_2}, v_3 = a^{-r_3}b^{s_3}, v_4 = a^{-r_4}b^{-s_4},$$ \ $$v_5 = a^{r_5}b^{s_5}, v_6 = a^{r_6}b^{-s_6}, v_7 = a^{-r_7}b^{s_7}, v_8 = a^{-r_8}b^{-s_8}.$$ (It suffices to take the sequences $(p_i)_{1\leq i\leq 8}, (q_i)_{1\leq i\leq 8}, (r_i)_{1\leq i\leq 8}, (s_i)_{1\leq i\leq 8}$ to be strictly increasing.) Let $A, B$ be these free groups generated by $u_1, \dots , u_8$ and $v_1, \dots , v_8$ respectively, and $\phi :A\to B$ be the isomorphism such that $\phi (u_i) = v_i, 1\leq i\leq 8.$ 
 
 \medskip
 
 Then, by Proposition \ref{thm:criterion}, the HNN extension $(G, A, B, t)$ where $t(a) = \phi (a)$ for all $a\in A$ is not left-orderable.   
 \end{proof}
 
 \begin{rem} Let us remind that in the case of rank = 1, the claim does not hold anymore since any HNN extension of $\Z $ is isomorphic $\langle t,a \ | \ ta^mt^{-1}=a^n\rangle $ for some non-zero integers $m,n$. All these groups (which include $\Z^2$, $\pi _1(\mathrm{Klein \ bottle}) = \langle a, b \ | \ aba^{-1} = b^{-1}\rangle $, and the solvable Baumslag-Solitar group $BS(1, n)\cong \Z\ltimes \Z[\frac{1}{n}]$), are all left-orderable as torsion-free one-relator groups.
 \end{rem}
 
 \medskip
 
 Using similar ideas, we build a non-left-orderable HNN extension of a left-orderable solvable group. We again rely on the criterion of Proposition \ref{thm:criterion}.

 \medskip
 
  Let $n\geq 2$ and $\Gamma _n$ be a group given by the presentation $$\langle s, x \ | \ [s^n,x]=1, [x, s^ixs^{-i}] = 1, 1\leq i\leq n-1 \rangle .$$
 
 Let $x_i = s^ixs^{-i}, i\in \mathbb{Z}$. Notice that $x_i = x_j$ iff $i\equiv j (\mod n)$. The elements $x_i, 0\leq i\leq n-1$ generate a normal subgroup $N_n$ isomorphic to $\mathbb{Z}^n$ and the quotient by this subgroup is isomorphic to $\mathbb{Z}$. Any element $g$ of $\Gamma _n$ can be written uniquely as $s^iw(x_0, \dots , x_{n-1})$ where $ i\in \mathbb{Z}$ and $w(x_0, \dots , x_{n-1}) = x_0^{p_0}\dots x_{n-1}^{p_{n-1}}$ for some integer exponents $p_0, \dots , p_{n-1}$. $s^iw(x_0, \dots , x_{n-1})$ will be called {\em the canonical form} of $g$. We also write $\Sigma (g) = i  +  p_0 + \dots + p_n$.
 
 \medskip
 
 Let us observe that $\Gamma _n$ is torsion-free. Indeed, if $g$ is a torsion element with canonical form $s^iw(x_0, \dots , x_{n-1})$ as above then for all $k\geq 1$, $$g^k = s^{ik}w_0(x_0, \dots , x_{n-1})w_i(x_0, \dots , x_{n-1})\dots w_{(k-1)i}(x_0, \dots , x_{n-1})$$ where $w_{j}(x_0, \dots , x_{n-1}) = w(x_j, x_{j+1}, \dots , x_{n-1+j})$ hence it follows immediately that either $i =0$; then, since $N_n\cong \mathbb{Z}^n$, we obtain that $w=1$.
 
 \medskip

 It turns out $\Gamma _n$ is left-orderable (which also implies that it is torsion-free). We  introduce a left order $<$ on $\Gamma _n$ as follows: An element $g$ with the canonical form $s^iw(x_0, \dots , x_{n-1})$ as above will be called positive if either $\Sigma (w) > 0$ or $\Sigma (w) = 0$ and $i > 0$.  If $\Sigma (w) = 0$ and $i = 0$, then we are in the group $N_n\cong \mathbb{Z}^n$ and there the order can be defined lexicographically. Then we see that a product of two positive elements is always positive and the inverse of a positive element is not positive. Hence $<$ is a left-order.
 
 \medskip
 
  To state our next proposition we need to introduce some (well-known) terminology. 
 
 \medskip
 
  \begin{defn} Let $G$ be a group generated by a subset $S\subseteq G\backslash \{1\}$ such that for all $x\in G$, if $x\in S$, then $x^{-1}\notin S$ (in particular, $1\notin S$). We say that a non-trivial reduced word $W(x_1, \dots , x_k) = x_1^{n_1}\dots x_k^{n_k}$ is positive in the alphabet $S$ if $x_1, \dots , x_k\in S$ and all exponents $n_i, 1\leq i\leq k$ are positive.  
  \end{defn}
  
  \medskip
  
  \begin{prop} \label{thm:gamma} In the group $\Gamma _n$ let $S_1 = \{s,x\}, S_2 = \{s^{-1},x\}, S_3 = \{s,x^{-1}\}, S_4 = \{s^{-1},x^{-1}\}$. For $n\geq 12$, there exists elements $f_1, \dots, f_4,$ \ $g_1, \dots , g_4\in \Gamma _n$ such that the following conditions hold:
  
  i) $\langle f_1, f_2, f_3, f_4\rangle \cong \langle g_1, g_2, g_3, g_4\rangle \cong \mathbb{Z}^4$, 
  
  ii) The elements $f_1, f_2, f_3, f_4$ can be represented with positive words in the alphabet $S_1$, 
  
  iii) For all $1\leq i\leq 4$, the element $g_i$ can be represented with a positive word in the alphabet $S_i$.
  \end{prop}
  
  \medskip
  
  \begin{proof} We define $f_1 = s^{n-1}xs, f_2 = s^{n-2}(xs)^2, f_3 = s^{n-4}(xs)^4, f_4 = s^{n-8}(xs)^8$. Then $f_1, f_2, f_3, f_4$ belong to $N_n$ and generate a subgroup isomorphic to $\mathbb{Z}^4$. We also define $g_1 = s^{n-1}xs, g_2 = s^{n-2}(x^{-1}s)^2, g_3 = s^{4-n}(xs^{-1})^4, f_4 = s^{8-n}(x^{-1}s^{-1})^8$. The elements $g_1, g_2, g_3, g_4$ also belong to $N_n$ and generate a subgroup isomorphic to $\mathbb{Z}^4$. 
  \end{proof}
  
  \medskip
  
  In the above proposition, the condition $n\geq 12$ is not necessarily the best possible. Using Proposition \ref{thm:gamma}, we can now prove the following proposition which establishes the existence of a non-left-orderable HNN extension of a left-orderable virtually Abelian group.
  
  \medskip
 
  \begin{prop} \label{thm:solvable} For all $n\geq 12$, $\Gamma _n$ admits a non-left-orderable HNN extension. 
  \end{prop}
  
  \begin{proof} Let $f_1, \dots, f_4, g_1, \dots , g_4\in \Gamma _n$ be elements satisfying conditions 1)-3) of Proposition \ref{thm:gamma}. Let $\phi :\langle f_1, f_2, f_3, f_4\rangle \to \langle g_1, g_2, g_3, g_4\rangle$ be an isomorphism such that $\phi (f_i) = g_i, 1\leq i\leq 4$. 
  
  We consider an HNN extension $$G:= (\Gamma _n, \langle f_1, f_2, f_3, f_4\rangle, \langle g_1, g_2, g_3, g_4\rangle, t )$$ by letting $txt^{-1} = \phi (x)$ for all $x\in  \langle f_1, f_2, f_3, f_4\rangle$. 
  
  For any left-order on $G$, notice that the elements $tf_it^{-1}, 1\leq i\leq 4$ are either all positive or all negative. On the other hand, among the elements $g_i, 1\leq i\leq 4$ at least one is positive and one is negative. This is a contradiction. Hence $G$ is not left-orderable. \end{proof}

 \section{HNN extensions of nilpotent groups}
 
 The aim of this section is to prove that, unlike solvable groups, an HNN extension of a left-orderable nilpotent group is always left-orderable. Let us recall that a nilpotent group is left-orderable iff it is torsion-free; this claim too does not hold for solvable groups. 
 
 \medskip

  We already observed that, by classification, an HNN extension of an infinite cyclic group is left-orderable. The same holds for an HNN extension of any torsion free Abelian group.

  \medskip

   Indeed, it suffices to consider finitely generated Abelian groups, so let $G$ be a finitely generated torsion-free Abelian group, $A, B\leq G$, $\phi :A\to B$ be an isomorphism, and $(G, A, B, t)$ be the HNN extension with respect to the isomorphism $\phi $. Let $G \cong \mathbb{Z}^d$ and $r = \mathrm{rank}A = \mathrm{rank}B$. We will assume that $G = \mathbb{Z}^d$. Then for some $1\leq r\leq d$ and linearly independent vectors $u_1, \dots , u_r$ we have $A = \{c_1u_1+\dots +c_ru_r \ : \ c_i\in \mathbb{Z}, 1\leq i\leq r\}$ and similarly for some linearly independent vectors $v_1, \dots , v_r$ we have $B = \{c_1v_1+\dots +c_rv_r \ : \ c_i\in \mathbb{Z}, 1\leq i\leq r\}$. We let $\overline {G} = \mathbb{R}^d, \overline{A} = \{c_1u_1+\dots +c_ru_r \ : \ c_i\in \mathbb{R}, 1\leq i\leq r\}\cong \R ^r, \overline{B} = \{c_1v_1+\dots +c_rv_r \ : \ c_i\in \mathbb{R}, 1\leq i\leq r\}\cong \R ^r$ and $\overline {\phi }:\overline{A}\to \overline{B}$ be the extension of $\phi :A\to B$ defined as $\overline{\phi }(c_1u_1+\dots +c_ru_r) = c_1\phi (u_1)+\dots + c_r\phi (u_r)$ for all $c_1, \dots , c_r\in \mathbb{R}$.     
  
  \medskip
  
   A key observation here is that even though the isomorphism $\phi :A\to B$ cannot necessarily be extended to $G$, but one can extend the isomorphism $\overline {\phi }:\overline{A}\to \overline{B}$ to some automorphism $F:\overline{G}\to \overline{G}$. Then the HNN extension $(\overline{G}, \overline{A}, \overline{B}, t)$ with respect to the isomorphism $\overline {\phi }:\overline{A}\to \overline{B}$ has a quotient isomorphic to the semidirect product $\mathbb{Z}\ltimes _{F}\overline{G}$ by a normal subgroup $N\trianglelefteq (\overline{G}, \overline{A}, \overline{B}, t)$. By the standard Bass-Serre theory, the normal subgroup $N$ is free because it acts freely in the Bass-Serre graph of the HNN extension (see Proposition 14 in \cite{S} and the discussion after that). Then $N$ is left orderable and since  $N$ and $\mathbb{Z}\ltimes _{F}\overline{G}$ are left-orderable we obtain that $(\overline{G}, \overline{A}, \overline{B}, t)$ is left-orderable (as an extension of a left-orderable group by a left-orderable group). On the other hand, $(\overline{G}, \overline{A}, \overline{B}, t)$ is a quotient of $(\overline{G}, A, B, t)$ by a normal subgroup, which is again by Bass-Serre theory, a free group. Thus, $(\overline{G}, A, B, t)$ is left-orderable.  By Britton's Lemma, $(G, A, B, t)$ is a subgroup of $(\overline{G}, A, B, t)$ hence it is also left-orderable.  
   
   \medskip
   
   We now would like to carry the same argument for any torsion-free nilpotent group. The main issue here is that given a finitely generated torsion-free nilpotent group $\Gamma $, one needs to construct a completion $\overline{\Gamma }$ that would resemble the operation $\mathbb{Z}^d\to \mathbb{R}^d$ so we can try to use the argument in the Abelian case. Before addressing this issue, let us observe the following lemma which, combined with a result of Karras-Solitar, allows to deduce quickly that HNN extensions of torsion-free nilpotent groups are left-orderable.

    \begin{lem}\label{thm:indicable} Let $\Gamma $ be a group such that an HNN extension of any finitely generated subgroup of $\Gamma $ is left-orderable. Then any HNN extension of $\Gamma $ is also left-orderable.
    \end{lem}

      \begin{proof} Since a direct limit of a left-orderable groups is left-orderable, it suffices to prove that for any countable subgroup $\Gamma '$, an HNN extension $(\Gamma ', A, B, t)$ of $\Gamma '$ is also left-orderable. Let $\Gamma _n, n\geq 1$ be subgroups of $\Gamma '$ such that $\Gamma _1\leq \Gamma _2\leq ...$, $\displaystyle \mathop{\cup }_{n\geq 1}\Gamma _n = \Gamma '$ and if $A_n = \Gamma _n\cap A$, then $\Gamma _n\cap B \supseteq tA_nt^{-1}$. Let $B_n = tA_nt^{-1}, n\geq 1$. Then $(\Gamma _m, A_n, B_n,t)$ is left-orderable by assumption, for all $m\geq n$. Hence, $(\Gamma ', A_n, B_n,t)$ is left-orderable. On the other hand, for all $n\geq 1$, $(\Gamma ', A_{n+1}, B_{n+1},t)$ is a quotient of $(\Gamma ', A_n, B_n,t)$ and $(\Gamma ', A, B,t)$ is obtained from $(\Gamma ', A_1, B_1,t)$ by successively adding all the relations of $(\Gamma ', A_n, B_n,t), n\geq 2$ and obtaining a left-orderable group at each step. Hence $(\Gamma ', A, B,t)$ is left-orderable. 

      \end{proof} 
   
   \medskip

   Now, to see that HNN extensions of torsion-free nilpotent groups are left-orderable, it remains to recall that an HNN extension of a finitely generated torsion-free nilpotent group is locally indicable (see \cite{KS}, Corollary on page 632), hence left-orderable, and then use Lemma \ref{thm:indicable}. We will take a different approach which allows to establish left-orderability of HNN extension of groups from a very broad class.  Namely, let $\mathcal{P}$ be the class of left-orederable groups $\Gamma $ such that for every finitely generated subgroup $G\leq \Gamma $ and for all finitely generated isomorphic subgroups $A, B\leq G$ with an isomorphism $\phi :A\to B$, the isomorphism $\phi $ extends to an isomorphism $\overline{\phi }:\overline{A}\to \overline{B}$, where $\overline{A},\overline{B}$ are subgroups of a left-orderable group $ \overline{G}$ such that $G\leq \overline{G}$ and the isomorphism $\overline{\phi }$ extends to an isomorphism $F:\overline{G}\to \overline{G}$. Let us emphasize that in this definition of class $\mathcal{P}$, we have mimicked the process of establishing that a torsion-free Abelian group belongs to $\mathcal{P}$ (and for torsion-free nilpotent groups, it will be a similar process). However, the definition of class $\mathcal{P}$ can be simplified as follows, by skipping the intermediate step groups $\overline{A}$ and $\overline{B}$: a left-orderable group $\Gamma $ is in class $\mathcal{P}$, if for every finitely generated subgroup $G\leq \Gamma $ and for all finitely generated isomorphic subgroups $A, B\leq G$ with an isomorphism $\phi :A\to B$, the isomorphism $\phi $ extends to an isomorphism $F:\overline{G}\to \overline{G}$, where $\overline{G}$ is left-orderable and $G\leq \overline{G}$. By a similar argument as in the case of torsion-free Abelian groups, we will show that for any group $\Gamma $ of the class $\mathcal{P}$, any HNN extension of $\Gamma $ is left-orderable, but most of our efforts will be on proving that torsion-free nilpotent groups belong to $\mathcal{P}$ (of course, the class $\mathcal{P}$ is very rich and much broader than torsion-free nilpotent groups).   

    \medskip

   Let $\mathcal{R}$ be a commutative ring with identity and $n\geq 1$. We let $U_n(\mathcal{R})$ be the group of $n\times n$  upper-triangular matrices with 1's on the diagonal. The cases $\mathcal{R} = \mathbb{R}$ and $\mathcal{R} = \mathbb{Z}$ will be the most interesting to us.
   
   \medskip
   
   It is well-known that any finitely generated torsion-free nilpotent group $\Gamma $ embeds in $U_n(\mathbb{Z})$ for some $n\geq 1$. The Mal'cev completion of $U_n(\mathbb{Z})$ is $U_n(\mathbb{R})$ and the Mal'cev completion of $\mathbb{Z}^n$ is $\mathbb{R}^n$ (See \cite{M, W})  \footnote{in the literature, the term {\em Mal'cev completion} is used for some other related operations as well.}, however, given an isomorphism $\phi :A\to B$ of subgroups of $U_n(\mathbb{Z})$, although it induces an isomorphism $\overline {\phi }:\overline{A}\to \overline{B}$ but one cannot necessarily extend this isomorphism to the entire $\overline{G}$. For example, for $n=3$, the group $U_3(\mathbb{Z})$ is isomorphic to the Heisenberg group $$\langle x, y, z \ | \ z=[x,y], [x,z]=[y,z]=1 \rangle $$ and if we let $A = \langle x \rangle , B = \langle z \rangle $ and $\phi (x) = z$, then this isomorphism cannot be extended to the isomorphism of $U_3(\mathbb{Z})$ (or $U_3(\mathbb{R}))$. Thus, we need to define a completion of $\Gamma $ other than the Mal'cev completion. 
   
   \medskip
   
    Let $X_{n,i}, 1\leq i\leq n-1$ be the matrix of $U_n(\mathbb{Z})$ where all off-diagonal entries are zero except that the $(i+1,i)$-th entry is equal to 1. In order to define a more suitable completion of  $U_n(\mathbb{Z})$ we will extend it first, and at the end we will obtain a completion which is "infinite-dimensional". Let $U_{\infty }(\mathbb{Z})$ be a group generated by $x_k, k\in \mathbb{Z}$ such that for all $k\in \mathbb{Z}, n\geq 1$ the subgroup generated by $x_{k+1}, \dots , x_{k+n-1}$ is isomorphic to $U_n(\mathbb{Z})$ through the isomorphism $f(x_{k+j}) = X_{n,j}, 1\leq j\leq n-1$. Notice that $U_{\infty }(\mathbb{Z})$ is well-defined in this way and contains isomorphic copies of all $U_n(\mathbb{Z}), n\geq 2$. This group can be viewed as the group of infinite sized integral unipotent matrices. But to achieve our goal, we extend $U_{\infty }(\mathbb{Z})$ further as follows. 
    
    \medskip
    
    Let us first observe that in the group $U_n(\mathbb{Z})$ viewed as the group of upper triangular unipotent integral matrices, $[x_i, x_j] = 1$ if $|i-j|\geq 2$ and for all $1\leq i\leq n-2$, $[x_i, x_{i+1}]$ is a unipotent matrix with all the off-diagonal entries zero, except the $(i+2, i)$-entry equals 1. Thus the elements $[x_i, x_{i+1}], 1\leq i\leq n-2$ generate a subgroup isomorphic to $U_{n-1}(\mathbb{Z})$ with an isomorphism $x_i\to [x_i, x_{i+1}], 1\leq i\leq n-2$. Similarly, in the group $U_{\infty }(\mathbb{Z})$, the elements $[x_i, x_{i+1}], i\in \mathbb{Z}$ generate a subgroup isomorphic to $U_{\infty }(\mathbb{Z})$, and the homomorphism $f:U_{\infty }(\mathbb{Z})\to U_{\infty }(\mathbb{Z})$ defined as $f(x_i) = [x_i, x_{i+1}], i\in \mathbb{Z}$ (it is sufficient to define it on the generators) establishes this isomorphism.     
    
    \medskip
    
    The group $U_{\infty }(\mathbb{Z})$ is a direct limit of the groups $U_n(\mathbb{Z}), n\geq 1$. More precisely, let $H_n, n\geq 1$ be the subgroup of $U_{\infty }(\mathbb{Z})$ generated by $x_{-n}, x_{-n+1}, \dots , x_{n-1}, x_n$. Then $H_n$ is isomorphic to $U_{2n+1}(\mathbb{Z})$, and $U_{\infty }(\mathbb{Z})$ is a direct limit of the sequence $H_n, n\geq 1$. Let also $K_n$ be the subgroup of $H_n$ generated by $x_1, \dots , x_n$. Then $K_n$ is isomorphic $U_n(\mathbb{Z})$.
    
\medskip 
    
    In our construction of the completion, we will use a direct limit of groups each isomorphic to $U_{\infty }(\mathbb{Z})$. Let $\Gamma _k, k\in \mathbb{Z}$ be a group generated by $z_{k,n}, n\in \mathbb{Z}$ with an isomorphism $g_k:\Gamma _k\to U_{\infty }(\mathbb{Z})$ such that $g_k(z_{k,n}) = x_n$. We have  $\dots \leq \Gamma _{-1}\leq \Gamma _0\leq \Gamma _1\leq \Gamma _2\leq \dots $ and $[z_{k,n}, z_{k,n+1}] = z_{k-1,n}$ for all $k, n\in \mathbb{Z}$. This defines an isomorphic embedding $g_{k,k+1}:\Gamma _k\to \Gamma _{k+1}, k\in \mathbb{Z}$ where $g_{k,k+1}(z_{k,n}) = z_{k+1,n}$. These inclusions define a direct limit $\mathcal{U}$ of $\Gamma _k, k\in \mathbb{Z}$. The maps $g_{k,k+1}$ induce a shift isomorphism $\Psi  : \mathcal{U}\to \mathcal{U}$, so, in particular, $\Psi (x) = g_{k,k+1}(x)$ for all $x\in \Gamma _k, k\in \mathbb{Z}$. For all $(r,s)\in \Z^2$, it is useful to consider $\theta ^{(r,s)}(z_{k,n}) = z_{k+r,n+s}$ and observe that $\theta ^{(r,s)}$ defines an isomorphism of $\mathcal{U}$ to itself. 
    
\medskip 

    In defining the completion $\overline{\mathcal{U}}$, first, let us recall the following facts about lattices of simply connected nilpotent Lie groups \cite{R}. 

\begin{prop} Let $G$ be simply connected nilpotent Lie group, $\Gamma $ be a discrete subgroup of $G$. The following are equivalent:

 (i) $\Gamma $ is a lattice of $G$;

 (ii) $\Gamma $ is Zariski dense in $G$;

 (iii) $\Gamma $ is not contained in any proper connected closed subgroup of $G$;

 (iv) $\Gamma $ is co-compact in $G$. 
\end{prop}

    \begin{defn} Let $m\geq 2$. For any subset $\Omega \subseteq U_{m}(\mathbb{Z})$, we define $Span(\Omega ) = \langle \Omega \rangle ^{Z}$ where the latter denotes the Zariski closure. For example, $Span ( U_{m}(\mathbb{Z})) =  U_{m}(\mathbb{R})$. Then, for any subset $\Omega \subseteq U_{\infty }(\mathbb{Z})$ we let $$Span(\Omega ) = \displaystyle \mathop{\cup }_{n\geq 1}Span(\Omega \cap H_n).$$ Then, for any subset $\Omega \subseteq \mathcal{U}$ we define $Span(\Omega ) = \displaystyle \mathop{\cup }_{k\geq 1}Span(\Omega \cap \Gamma _k).$ Finally, we define $\overline {\mathcal{U}} = Span (\mathcal{U})$. 
\end{defn}
    
    \medskip

 The Lie subgroups of $U_n(\mathbb{R})$ (hence of $\overline{\mathcal{U}}$) are simply connected (indeed, contractible, as the exponential map determines a homeomorphism to $\R^d$ with $d$ being the dimension of the group) thus its isomorphism type can be determined at the level of Lie algebras. The Lie algebra of every Lie subgroup of $\overline{\mathcal{U}}$ is a finite-dimensional nilpotent Lie algebra. On the other hand, by Engel's Theorem, for every finite-dimensional nilpotent Lie algebra $\mathbf{g}$ with the underlying vector space $V$, there exists {\em an associated flag} $\mathcal{F}(\mathbf{g})$ in the form $\{0\} = V_0\leq V_1\leq \dots \leq V_n = V$ where $\dim V_i = i, 0\leq i\leq n$ and for all $x\in \mathbf{g}, 1\leq i \leq n$, $ad(x)(V_i)\subseteq V_{i-1}$. Thus, $\mathbf{g}$ can be faithfully represented by strictly upper-triangular matrices with respect to some basis of $V$. If $\mathbf{g}, \mathbf{h}$ are finite-dimensional nilpotent Lie algebras and $\phi :\mathbf{g}\to \mathbf{h}$ a Lie algebra isomorphism, then $\mathcal{H} = \phi (\mathcal{F})$ will be an associated flag of $\mathbf{h}$. On the other hand, if $\mathbf{g}$ is a finite-dimensional nilpotent Lie algebra with underlying vector space $V$ and $I$ is an ideal of $\mathbf{g}$ faithfully represented in $\mathbf{gl}(V_0)$ with strictly upper triangular matrices with respect to a basis of a proper subspace $V_0$,  then by inductive process as in the proof of Engels' Theorem, it follows that we can extend the basis of $V_0$ to a basis of $V$ such that $\mathbf{g}$ is faithfully represented with strictly upper triangular matrices. Based on this observation, one can show that any Lie group isomorphism $\Phi :G\to H$ between finite-dimensional nilpotent Lie subgroups of $\overline{\mathcal{U}}$ can be extended to the group automorphism of $\overline{\mathcal{U}}$. We will provide a direct and explicit proof of this extension. First, let the Lie algebra \( \overline{\mathfrak{u}} \) as the Lie algebra of \emph{finitary strictly upper-triangular infinite matrices} indexed by the integers \( \mathbb{Z} \):
\[
\overline{\mathfrak{u}} := \left\{ X = (x_{ij})_{i,j \in \mathbb{Z}} \; \middle| \;
\begin{array}{l}
x_{ij} = 0 \text{ whenever } i \ge j, \text{and } X \text{ has finite} \\
\text{support, i.e., } \# \{ (i,j) : x_{ij} \neq 0 \} < \infty
\end{array}
\right\}.
\]

In other words, \( \overline {\mathfrak{u}} \) is the union of all finite-dimensional strictly upper-triangular matrix Lie algebras
\[\overline{\mathfrak{u}} = \bigcup_{\substack{a,b \in \mathbb{Z} \\ a < b}} \mathfrak{u}_{[a,b]},
\]
where
\[
\mathfrak{u}_{[a,b]} := \left\{ X \in \overline{\mathfrak{u}} \; \middle| \; x_{ij} = 0 \text{ if } i,j \notin [a,b] \right\}.
\]

This definition allows embeddings of finite-dimensional Lie algebras \( \mathfrak{u}_n \):=  \( \mathfrak{u}_{I_n}\) with \(I_n = \{1, 2, \dots , n\} \) into \( \overline{\mathfrak{u}} \) supported on arbitrary finite intervals of \( \mathbb{Z} \), providing a universal environment for all finite-dimensional strictly upper-triangular Lie algebras. In the proof of the following proposition, we follow the technique used in \cite{Be}.

 \medskip 
 
  \begin{prop}\label{prop:extension}  Let  \( \mathfrak{g}, \mathfrak{h} \subseteq \mathfrak{u}_n \) be finite-dimensional nilpotent Lie algebras embedded in \(\overline{ \mathfrak{u}} \), and let \( \phi : \mathfrak{g} \to \mathfrak{h} \) be a Lie algebra isomorphism. Then there exists an automorphism \( \Phi \in \mathrm{Aut}(\overline{\mathfrak{u}}) \) such that
\[
\Phi|_{\mathfrak{g}} = \phi.
\]
\end{prop}

\begin{proof}
Since \( \mathfrak{g} \) and \( \mathfrak{h} \) are finite-dimensional and isomorphic, choose disjoint finite index intervals \( I = [i_1, i_2] \subset \mathbb{Z} \) and \( J = [j_1, j_2] \subset \mathbb{Z} \), both of length \( n \), such that $|x-y| > n$ for all $x\in I, y\in J$ and 
\[
\mathfrak{g} \hookrightarrow \mathfrak{u}_I,\quad \mathfrak{h} \hookrightarrow \mathfrak{u}_J,
\]
where \( \mathfrak{u}_I \subset \overline{\mathfrak{u}} \) denotes the subalgebra of matrices supported on rows and columns in \( I \), and similarly for \( \mathfrak{u}_J \).

\medskip 

Define \( \iota_{\mathfrak{g}} : \mathfrak{g} \hookrightarrow \mathfrak{u}_I \) and \( \iota_{\mathfrak{h}} : \mathfrak{h} \hookrightarrow \mathfrak{u}_J \) as the chosen embeddings and let $\psi : \iota_{\mathfrak{g}}(\mathfrak{g})\to \iota_{\mathfrak{h}}(\mathfrak{h})$ defined as $\psi =                \iota_{\mathfrak{h}}\circ \phi \circ \iota_{\mathfrak{g}}^{-1}$. It suffices to prove that $\psi $ can be lifted to an isomorphism of \(\overline{ \mathfrak{u}} \). 

\medskip 

 Since the index sets $I$ and $J$ have length $n$ and are at a distance at least $n$, the isomorphism $\psi $ can be lifted to $\psi ':\mathfrak{u}_I\to \theta ^{k,l}(\mathfrak{u}_J)$ for some $0\leq k, l\leq n-1$.  

\medskip 

 Let also $\mathcal{B}_0, \mathcal{B}_1, \mathcal{B}$ be bases of $\iota_{\mathfrak{g}}(\mathfrak{g}), \mathfrak{u}_I, \overline{\mathfrak{u}}$ respectively, such that $\mathcal{B}_0\subseteq \mathcal{B}_1\subset \mathcal{B}$. Then we extend the isomorphism $\psi '$ (which is already defined in $\mathcal{B}_1$) inductively along the basis $\mathcal{B}$ and obtain an isomorphism $\Psi : \overline{\mathfrak{u}}\to \overline{\mathfrak{u}}$. This extension is straightforward if we use the following standard fact: For all $k\geq 1$, let $\mathcal{N}_k$ denote the Lie algebra of $k\times k$ real strictly upper-triangular matrices. $\mathcal{N}_k$ has a standard embedding in $\mathcal{N}_{k+1}$ as the sub-algebra of upper-left block of size $k\times k$.  Then any Lie algebra automorphism $\mathcal{N}_k\to \mathcal{N}_k$ lifts to an automorphism $\mathcal{N}_{k+1}\to \mathcal{N}_{k+1}$. Indeed, this fact also suggests a specific natural order to extend the basis $\mathcal{B}_1$ to $\mathcal{B}$.

\end{proof}
  
\medskip 
    
    We can now state and prove the following.
    
    \begin{thm} \label{thm:nilpotent} a) An HNN extension of any group of class $\mathcal{P}$ is left-orderable; 
    
     b) Any torsion-free nilpotent group belongs to class $\mathcal{P}$.
    \end{thm}
    
    \medskip
    
    \begin{proof} We will start with the proof of part b). Let $\Gamma $ be a torsion-free nilpotent group. It is well-known that $\Gamma $ is left-orderable (in fact, bi-orderable). Indeed, it suffices to prove this only for finitely generated subgroups, and any such subgroup embeds into $U_m(\mathbb{Z})$ for some $m\geq 2$. The latter admits an easy bi-order. Indeed, more generally, we define a matrix $A = (a_{i,j})_{1\leq i, j\leq n}\in U_m(\mathbb{R})$ as positive if $d$ is the smallest positive integer such that $a_{i,j} \neq 0$, for some $i, j\geq 1$ with $i+j = d$, moreover, for this $d$, if $p$ is the smallest positive integer with $p+q = d$ and $a_{p,q}\neq 0$, then $a_{p,q} > 0$. One easily checks that this is in fact a genuine left-order (and even a bi-order). Then $U_{\infty }(\mathbb{R})$ is also bi-orderable as a direct limit of $U_{m}(\mathbb{R}), m\geq 1$ and so is $\overline{\mathcal{U}}$.
    
    \medskip
    
    Let us assume that $\Gamma $ is finitely generated, $A, B\leq \Gamma $ are finitely generated subgroups of $\Gamma $ and $\phi :A\to B$ be an isomorphism. $\Gamma $ embeds in $U_{\infty }(\mathbb{Z})$ and the latter is a subgroup of $G = U_{\infty }(\mathbb{R})$. 
    
    \medskip
    
    The isomorphism $\phi :A\to B$ cannot necessarily be extended to $G$, but one can extend the isomorphism $\overline {\phi }:Span(A)\to Span(B)$ to some $F:\overline{\mathcal{U}}\to \overline{\mathcal{U}}$ where $\overline {\phi }$ is an extension of $\phi $ by Mostow Strong Rigidity Theorem for lattices in solvable Lie groups \cite{R} and by Proposition \ref{prop:extension}.
    
    \medskip 

    Now, for part a), let $\Gamma $ be a group of class $\mathcal{P}$. We let $G$ be a finitely generated subgroup of $\Gamma $, $A, B$ be finitely generated subgroups of $G$ with an isomorphism $\phi:A\to B$ and its extension $\overline{\phi }:\overline{A}\to \overline{B}$ such that $\overline{A}, \overline{B}$ are subgroups of a left-orderable group $\overline{G}$ such that $\overline{G}$ contains $G$ as a subgroup and $\overline{\phi }$ extends to an isomorphism $F:\overline{G}\to \overline{G}$. 

    \medskip 

     Then the HNN extension $(\overline{G}, \overline{A}, \overline{B}, t)$ with respect to the isomorphism $\overline {\phi }:\overline{A}\to \overline{B}$ has a quotient isomorphic to the semidirect product $\mathbb{Z}\ltimes _{F}\overline{G}$ by a free normal subgroup $N\trianglelefteq (\overline{G}, \overline{A}, \overline{B}, t)$ (the normal subgroup $N$ is free again by the Bass-Serre theory). Since $N$ and $\mathbb{Z}\ltimes _{F}\overline{G}$ are left-orderable we obtain that $(\overline{G}, \overline{A}, \overline{B}, t)$ is left-orderable (as an extension of a left-orderable group by a left-orderable group). Then, $(\overline{G}, A, B, t)$ is also left-orderable as an extension of $(\overline{G}, \overline{A}, \overline{B}, t)$ by a free normal subgroup. By Britton's Lemma, $(\Gamma , A, B, t)$ is a subgroup of $(\overline{G}, A, B, t)$ hence it is also left-orderable. 
     
    \end{proof}

 \medskip
 
 We would like to end this section with a torsion-free non-left-orderable example which will contain a class two nilpotent group as an index two subgroup; indeed, it will contain a subgroup of Heisenberg group $H$ of $3\times 3$ integral unipotent matrices. This might be one of the simplest (smallest) examples of a torsion-free non-left-orderable group in the literature. (Another example is the Hantzsche-Wendt group, a specific crystallographic group that is virtually $\Z ^3$; see \cite{BRW}) It also reaffirms that torsion-freeness does not imply left-orderability in the class of polycyclic groups. 
 
 \medskip
 
 Let $$\Gamma = \langle t, u, v \ | \ [u,v] = t^4, tut^{-1} = u^{-1}, tvt^{-1} = v^{-1}\rangle .$$
 
 The group $\Gamma $ is related to the Heisenberg group $$H = \langle x, y, z \ | \ [x, y] = z, [z, x] = [z, y] = 1\rangle .$$
 
 Any element of $H$ can be written uniquely as $x^my^nz^k$ where $m, n, k\in \Z$. $H$ is bi-orderable. The elements $x^2, y, z$ generate an index two subgroup $H_0$ of $H$. 
 
 \medskip
 
 The group $\Gamma $ will have an index-two subgroup isomorphic to $H_0$. We let $u = x^2, v = y, t^2 = z$. Let also $G$ be a group given by the following presentation $$G = \langle t, x, y, z \ | \ [x, y] = z, [z, x] = [z, y] = 1, t^2 = z, txt^{-1} = x^{-1}, tyt^{-1} = y^{-1}\rangle .$$ Then $\Gamma $ is a subgroup of $G$ generated by $t, x^2, y$.
 
 \medskip
 
 \begin{prop} $\Gamma $ is torsion-free and non-left-orderable.  
 \end{prop}
 
 \begin{proof} Assume that $<$ is a left-order on $\Gamma $. Without loss of generality, we may assume that $t > 1$. Then $1 < t < z$ thus $z$ is also a positive element. On the other hand, let us observe that for all $n\in 2\Z $, we have $(tx^n)^2 = t^2$ thus the element $tx^n$ is positive for every even integer $n$. Let $m$ be positive if $y > 1$ and negative if $y < 1$. Then, for all $m\in 2\Z$, $tx^ny^m$ is positive as a product of two positive elements $tx^n$ and $y^m$. Then $(tx^ny^m)^2 > 1$. However, $$(tx^ny^m)^2 = t^2x^{-n}y^{-m}x^ny^m = t^2z^{mn} = z^{mn+1}.$$
 
 We can choose $n$ so that $mn + 1 < 0$. This yields $(tx^ny^m)^2 < 1$. Contradiction.
 
 \medskip
 
 To see torsion-freeness let us observe that any element $g\in G$ can be written as $g = x^{2p}y^qz^r$ or $g = tx^{2p}y^qz^r$. The element $x^{2p}y^qz^r$ is not a torsion, since $H$ is torsion-free. As for the element $tx^{2p}y^qz^r$, we have $(tx^{2p}y^qz^r)^2 = x^{-2p}y^{-q}x^{2p}y^qz^{2r+1} = z^{\pm 2pq}z^{2r+1} \neq 1.$ Thus, $\Gamma $ is torsion-free.
  \end{proof}
  
 \medskip

{\em Acknowledgment: } We are very thankful to Zipei Nie for reading the draft of this paper and for correcting the errors.

 \end{document}